\documentclass{amsart}  
\usepackage{amssymb, tikz, float, enumitem, mathtools}
\usetikzlibrary{calc,decorations.markings,matrix,arrows,positioning}
\tikzset{>=latex}
\usepackage[breaklinks, final]{hyperref}

\newcommand{\A}{\mathbb{A}}
\newcommand{\m}[1]{\mathbb{#1}}    
\newcommand{\cl}[1]{\mathcal{#1}}  

\theoremstyle{plain} \newtheorem{thm}{Theorem}[section]
\newtheorem{prop}[thm]{Proposition}
\newtheorem{lem}[thm]{Lemma}

\theoremstyle{definition} \newtheorem{defn}[thm]{Definition}
\theoremstyle{remark} \newtheorem*{rk}{Remark}
\theoremstyle{plain} \newtheorem*{claim}{Claim}
\newcommand{\Case}[1]{\smallskip \textbf{Case #1:}}
\newenvironment{claimproof} {
  \begin{proof}[Proof of claim]
  
  } {
  \end{proof}
  }

\DeclareMathOperator{\Sg}{Sg}
\DeclareMathOperator{\Con}{Con}

\numberwithin{equation}{section}  
\renewcommand{\phi}{\varphi}
\renewcommand{\epsilon}{\varepsilon}

\usepackage{caption, subcaption, xparse}

\theoremstyle{definition} 
\theoremstyle{remark} 

\DeclareMathOperator{\lines}{Lines}
\DeclareMathOperator{\squares}{Squares}

\DeclareMathOperator{\gcube}{gCube}

\def\nat{\mathbb{N}}
\def\A{\mathbb{A}}
\def\var{\mathcal{V}}

\def\Meet{\bigwedge}
\def\Union{\bigcup}

\def\union{\cup}

\tikzset{myStyle/.style={baseline=(center.base), font=\small,
    every node/.style={inner sep=0.25em} }}

\NewDocumentCommand{\LinePic}{ O{} O{} O{1} }{ 
  \begin{tikzpicture}[myStyle, scale=#3*1 ]
    \node (center) at (0,0.5) {\phantom{$\cdot$}}; 
    \path (0,0)  node (s) {$#1$}
        ++(0,1)  node (n) {$#2$};
    \draw (n) -- (s);
  \end{tikzpicture}
}  

\newcommand{\SquareUnwrapped}[4]{ 
  \node (center) at (0.5,-0.5) {\phantom{$\cdot$}}; 
  \path (0,0)  node (nw) {$#2$}
      ++(1,0)  node (ne) {$#4$}
      ++(0,-1) node (se) {$#3$}
      ++(-1,0) node (sw) {$#1$};
  \draw (nw) -- (ne) -- (se) -- (sw) -- (nw);
}  
\NewDocumentCommand{\SquareXY}{ O{} O{} O{} O{} O{1} O{1} }{ 
  \begin{tikzpicture}[myStyle, xscale=#5*1, yscale=#6*1 ]
    \SquareUnwrapped{#1}{#2}{#3}{#4}
  \end{tikzpicture}
}  
\NewDocumentCommand{\Square}{ O{} O{} O{} O{} O{1} }{ 
  \SquareXY[#1][#2][#3][#4][#5][#5]
}  

\NewDocumentCommand{\SquareAxes}{ O{} O{} O{1} O{0} O{1} }{  
  \begin{tikzpicture}[myStyle, scale=#3*0.8] 
    \node (center) at (0.5,0.5) {\phantom{$\cdot$}}; 
    \draw (0,0) -- node[above]{$#1$} (1,0) node[right]{$#4$}
      (0,0) -- node[left]{$#2$} (0,1) node[above]{$#5$};
  \end{tikzpicture}
}  
\NewDocumentCommand{\CubeAxes}{ O{} O{} O{} O{1} O{0} O{1} O{2} }{  
  \begin{tikzpicture}[myStyle, scale=#4*0.85] 
    \node (center) at (0.5,0.75) {\phantom{$\cdot$}}; 
    \draw (0,0) -- node[above]{$#1$} (1,0) node[right]{$#5$}
      (0,0) -- node[left]{$#2$} (0,1) node[above]{$#6$}
      (0,0) -- node[below left=-0.25em]{$#3$} (0.5,-0.5) node[below right=-0.2em]{$#7$};
  \end{tikzpicture}
}  

\newcommand{\CubeNodes}[8]{  
  \node at (0.75,-0.75) (center) {\phantom{$\cdot$}}; 
  \path (0,0)  node (back_nw)      {$#2$}
      ++(1,0)  node (back_ne)      {$#4$}
      ++(0,-1) node (back_se)      {$#3$}
      ++(-1,0) node (back_sw)      {$#1$}
        (0.5,-0.5) node (front_nw) {$#6$}
      ++(1,0)  node (front_ne)     {$#8$}
      ++(0,-1) node (front_se)     {$#7$}
      ++(-1,0) node (front_sw)     {$#5$};
}  
\newcommand{\CubeUnwrapped}[8]{ 
  \CubeNodes{#1}{#2}{#3}{#4}{#5}{#6}{#7}{#8}
  \draw (back_nw) -- (back_ne) -- (back_se) -- (back_sw) -- (back_nw)
    (front_nw) -- (front_ne) -- (front_se) -- (front_sw) -- (front_nw)
    (back_nw) -- (front_nw)
    (back_ne) -- (front_ne)
    (back_se) -- (front_se)
    (back_sw) -- (front_sw);
}  
\newcommand{\CubeDUnwrapped}[8]{ 
  \CubeNodes{#1}{#2}{#3}{#4}{#5}{#6}{#7}{#8}
  \draw (front_nw) -- (front_ne) -- (front_se) -- (front_sw) -- (front_nw)
    (back_nw) -- (back_ne) (back_sw) -- (back_nw)
    (back_nw) -- (front_nw)
    (back_ne) -- (front_ne)
    (back_sw) -- (front_sw);
  \draw[densely dotted] (back_ne) -- (back_se) -- (back_sw)
    (back_se) -- (front_se);
}  
\NewDocumentCommand{\Cube}{ O{} O{} O{} O{} O{} O{} O{} O{} O{1} }{  
  \begin{tikzpicture}[myStyle, scale=#9*1 ]
    \CubeUnwrapped{#1}{#2}{#3}{#4}{#5}{#6}{#7}{#8}
  \end{tikzpicture}
}  
\NewDocumentCommand{\CubeD}{ O{} O{} O{} O{} O{} O{} O{} O{} O{1} }{  
  \begin{tikzpicture}[myStyle, scale=#9*1 ]
    \CubeDUnwrapped{#1}{#2}{#3}{#4}{#5}{#6}{#7}{#8}
  \end{tikzpicture}
}  

\NewDocumentCommand{\DeltaZeroCubeD}{ O{} O{} O{} O{} O{} O{} O{} O{} O{1} }{  
  \begin{tikzpicture}[myStyle, scale=#9*1]
    \CubeDUnwrapped{#1}{#2}{#3}{#4}{#5}{#6}{#7}{#8}
    \draw (back_sw)  to[out=30,in=180-30] (back_se)
      (back_nw)  to[out=30,in=180-30] node[auto]{$\delta$} (back_ne)
      (front_sw) to[out=30,in=180-30] (front_se);
    \draw[dashed] (front_nw) to[out=30,in=180-30] (front_ne);
  \end{tikzpicture}
}  
\NewDocumentCommand{\DeltaTwoCubeD}{ O{} O{} O{} O{} O{} O{} O{} O{} O{1} }{  
  \begin{tikzpicture}[myStyle, scale=#9*1]
    \CubeDUnwrapped{#1}{#2}{#3}{#4}{#5}{#6}{#7}{#8}
    \draw (back_sw) to[out=180+30,in=180] node[auto,swap]{$\delta$} (front_sw)
      (back_se) to[out=0,in=30] (front_se)
      (back_nw) to[out=180+30,in=180] (front_nw);
    \draw[dashed] (back_ne) to[out=0,in=30] (front_ne);
  \end{tikzpicture}
}  

\DeclareMathOperator{\M}{\text{M}}  
\DeclareMathOperator{\C}{\text{C}}  
\DeclareMathOperator{\SLines}{S-Lines}  

\begin{document}
\title{Supernilpotence Need Not Imply Nilpotence}
\author{Matthew Moore, Andrew Moorhead}
\date{\today}

\address[Matthew Moore]{
  University of Kansas
  Dept.\ of Electrical Engineering and Computer Science;
  Eaton Hall;
  Lawrence, KS 66044;
  U.S.A.}
\email[Matthew Moore]{matthew.moore@ku.edu}
\address[Andrew Moorhead]{
  Department of Mathematics;
  Vanderbilt University;
  Nashville, TN;
  U.S.A.}
\email[Andrew Moorhead]{andrew.p.moorhead@vanderbilt.edu}

\thanks{The second author was supported by the National Science Foundation
  grant no.\ DMS 1500254}

\begin{abstract}
Supernilpotence is a generalization of nilpotence using a recently developed
theory of higher-arity commutators for universal algebras. Many important
structural properties have been shown to be associated with supernilpotence,
and the exact relationship between nilpotence and supernilpotence has been
the subject of investigation. We construct an algebra which is not solvable
(and hence not nilpotent) but which is supernilpotent, thereby showing that
in general supernilpotence does not imply nilpotence. We also extend this
construction to `higher dimensions' to obtain similar results for $(n)$-step
supernilpotence.
\end{abstract} 
\maketitle 

\section{Introduction} \label{sec:intro} 
The topic of this manuscript is related to a broad generalization of
commutator theory called higher commutator theory. Higher commutators are
used to define a condition called \emph{supernilpotence}, called such
because it is usually a stronger condition than nilpotence. We construct
algebras to demonstrate that supernilpotence and nilpotence are in general
independent of one another and that this independence is preserved even if
one considers `higher dimensional' analogues of nilpotence. 

Historically, a specific notion of commutator was used used to study a
specific variety of algebras, (e.g.\ a class of similar algebraic structures
that satisfy some set of equational laws or identities), such as the variety
of groups, rings, or Lie algebras. In each of these classes the notion of a
commutator has led to important structural results, as it can be used to
measure `abelianness' and define generalizations of abelianness such as
solvability and nilpotence. For example, a classical theorem of group theory
states that a finite group is nilpotent if and only if it is the direct
product of its Sylow subgroups. 

Actually, each of these commutator theories is a special case of a
commutator that may be formulated for any algebraic structure. The strength
of the theory depends not on the similarity type of the algebra, but on the
identities that it satisfies. The initial insight is due to Smith. In
\cite{jdhsmith}, he shows it is possible to define a commutator for any
variety of algebras in which every member has a Mal'cev operation, that is,
an operation $p(x,y,z)$ built from the basic operations that satisfies the
identities
\[
  p(x,x,y)
  \approx p(y,x,x)
  \approx y,
\]
and that this commutator retains all the essential features of the examples
known at the time, all of which were for algebras with a Mal'cev operation.

Commutator theory for universal algebras has grown substantially since then
and we do not attempt a survey in this introduction. We refer the reader to
the text Freese and McKenzie \cite{fm} and the text Gumm \cite{gumm} for two
different approaches to commutator theory for congruence modular varieties
of algebras. For the development of commutator theory outside of the context
of congruence modularity, the reader is referred to the monograph Kiss and
Kearnes \cite{kearneskiss}.

Such a general commutator theory comes equipped with the naturally
generalized versions of abelianness, solvability, and nilpotence. Under some
additional assumptions, finite nilpotent algebras are very similar in their
structure to finite nilpotent groups. For example, Lyndon
\cite{lyndonnilgroup} shows that the equational theory of a nilpotent group
is finitely based and Freese and McKenzie \cite{fm} shows that if a finite
algebra of finite type (belonging to a congruence modular variety) is
nilpotent \emph{and} is the direct product of nilpotent algebras of prime
power order, then it has a finitely based equational theory. Such algebras
are now known to be examples of \emph{supernilpotent} algebras.

Supernilpotence is an analogue of abelianness that is definable with a
higher arity commutator that generalizes the classical binary commutator.
Such commutators were first introduced by Bulatov in \cite{buldef}. In
\cite{aichmud}, Aichinger and Mudrinksi develop analogues of those
properties shown to be essential for the binary commutator for the higher
commutator (in a Mal'cev variety). In the same paper every supernilpotent
algebra belonging to a Mal'cev variety is shown to be nilpotent. Using
earlier results of Kearnes from \cite{smallfreespec}, Aichinger and
Mudrinksi go on to prove that every finite supernilpotent Mal'cev algebra of
finite type is a product of prime power order nilpotent Mal'cev algebras,
and vice versa.

Supernilpotent Mal'cev algebras of finite type share other properties with
nilpotent groups. For example, Michael Kompatscher shows in
\cite{kompatscherSupNil} that there is a polynomial time algorithm that
checks if equations over finite supernilpotent Mal'cev algebras of finite
type have a solution. Equation solvability and related problems emphasize
the need to understand the differences between nilpotence and
supernilpotence, see Idziak and Krzaczkowski \cite{IdziakKrz} for additional
details.

The theory of the higher commutator has been recently extended to varieties
that are not Mal'cev. In \cite{moorheadHC}, the second author extends most
of the theory of the higher commutator to congruence modular varieties. In
\cite{delta3v}, the second author develops a relational description of the
modular ternary commutator and uses this to show that $(2)$-step
supernilpotence implies $(2)$-step nilpotence in a congruence modular
variety. In Wires \cite{Wires}, several properties of higher commutators are
developed outside of the context of congruence modularity. Implicit in the
results of Wires is that supernilpotence implies nilpotence for congruence
modular varieties. More recently, Kearnes and Szendrei have announced that
any \emph{finite} supernilpotent algebra is nilpotent, which is to appear
appear in \cite{finitesupnil}. It turns out that supernilpotence is a
stronger condition than nilpotence for any variety of algebras that
satisfies a nontrivial idempotent equational condition. This result will
appear in \cite{taylorsupnil}. 

Each of the algebras we construct in this paper is therefore infinite and
does not generate a Taylor variety. In Section \ref{sec:defs} we develop
notation and state definitions. In Section \ref{sec:generalized_nil} we
discuss different notions of nilpotence and solvability. In Section
\ref{sec:algebra_A2} we construct an algebra that is not solvable but is
supernilpotent. The final section \ref{sec:algebra_An} generalizes this
example to `higher dimensions'.

\section{Definitions} \label{sec:defs} 
\subsection{Notation}
In this paper the set of natural numbers is denoted by $\omega$ and has as
its least element the empty set, or $0$. The finite ordinal $n$ is the set
of its predecessors and we will often write $i\in n$ instead of $0 \leq i
\leq n-1$.

Some familiarity with the basics of Universal Algebra is assumed. Good
references on the subject are \cite{burrsan} and \cite{alv}. An
\textbf{algebra} is a set with some structure provided by a set of finitary
operations. These two ingredients are usually written as a pair, e.g.\ $\A =
\left< A; \{f_i\}_{i\in I} \right>$. Product, subalgebra, and homomorphism
are defined in the obvious way.

Let $\A$ be an algebra, $n\in \omega$, and $R \subseteq A^n$ be a set of
tuples over $A$ of length $n$. If $R$ is a subalgebra of $\A^n$ we say that
$R$ is an \textbf{$\A$-invariant} relation, or just an invariant relation if
there is no possibility for confusion.

The invariant equivalence relations of an algebra are called
\textbf{congruences} and determine its possible homomorphic images. The
lattice of all congruences of an algebra is denoted by $\Con(\A)$, with the
largest congruence and least congruence denoted by $1$ and $0$,
respectively.

\subsection{The Higher Commutator}
The higher commutator is an operation on the lattice of congruences of an
algebra and is usually defined via the so-called term condition, see
\cite{buldef} for the first instance in the literature. The main
construction of this paper is most naturally presented by defining the
commutator via a special invariant relation which we now describe. The
commutator definition given here is equivalent to the usual one and the
reader is referred to \cite{moorheadHC} or \cite{orsalrel} for more details. 

Let $\A = \left< A; \{f_i\}_{i\in I} \right>$ be an algebra and $n\in
\omega$ a natural number. An invariant relation 
\[
  R \leq \A^{2^n}
\]
is said to be an \textbf{$(n)$-dimensional invariant relation}. The reason
for this terminology is that the set of functions $2^n$ is a natural
coordinate system for the $(n)$-dimensional cube, where two functions are
connected by an edge if and only if they differ in exactly one argument. A
particular element
\[
  h \in A^{2^n}
\]
is therefore thought of as a \textbf{vertex labeled $(n)$-dimensional cube}.
Less formally, we will sometimes refer to $h$ simply as an $(n)$-dimensional
cube, or (when the dimension is clear) a just a cube.

A total function $f \in 2^n$ specifies the coordinates of a particular
vertex of such an $h \in A^{2^n}$, and we denote the value of $h$ at $f$ by
$h_f$. This notation may be extended to partial functions, and in doing so
one may specify inside of $h$ the location of lower dimensional vertex
labeled cubes. That is, for $S \subseteq n$ and $f: S \to 2$ define
\[
  h_f \coloneqq \big\{ h_g : g\in 2^n \text{ and } f \subseteq g \big\}.
\]
Less formally, a partial function $f: n \to 2$ determines some of the
coordinates for a vertex in $h$. The coordinates that are not yet determined
may be specified by those $g \in 2^n$ that extend $f$. We hope that the
reader alarmed by the potential ambiguity of this notation will find no
ambiguity in its use.

We distinguish for any domain $S$ the function $\textbf{1}:S \to 2$ that
takes the constant value $1$. Take some $h\in A^{2^n}$ and $i \neq j \in n$.
Define 
\begin{align*}
  \lines_{i}(h) 
    &\coloneqq \left\{ h_f : f\in 2^{n\setminus\{i\}} \right\} \\
    &= \left\{ h_f: f\in 2^{n\setminus\{i\}}, f\neq \textbf{1} \right\} 
      \union \left\{ h_{\textbf{1}} : \textbf{1} \in 2^{n\setminus\{i\}} \right\}
      \text{ and } \\
  \squares_{i,j}(h) 
    &\coloneqq \left\{ h_f : f\in 2^{n\setminus\{i,j\}} \right\} \\
    &= \left\{ h_f: f\in 2^{n\setminus\{i,j\}}, f\neq \textbf{1} \right\}
      \union \left\{ h_{\textbf{1}} : \textbf{1} \in 2^{n\setminus\{i,j\}} \right\}.
\end{align*}
These sets are called the \textbf{$(i)$-cross section lines} and the
\textbf{$(i,j)$-cross section squares} of $h$, respectively. The set of
$(i)$-cross section lines is the disjoint union of two sets. The first set
we denote by $\SLines_i(h)$ and its members are called \textbf{$(i)$-support
lines}; the single member of the second set is called the
\textbf{$(i)$-pivot line}. Similarly, the set of $(i,j)$-cross section lines
is composed of \textbf{$(i,j)$-support squares} and a single
\textbf{$(i,j)$-pivot square}. See Figure \ref{fig:lines_squares}. We say
that a line, square, or (generally) a cube is \textbf{constant} if all of
the vertices have the same value. We call a set of lines, squares, or cubes
\textbf{constant} if all of its members are.

\begin{figure}  
\def\myTikzScale{1}   
\begin{subfigure}{0.15\textwidth} \centering
  \CubeAxes[][][][\myTikzScale]
\end{subfigure}%
\begin{subfigure}{0.85\textwidth}
  \begin{align*}
    \lines_2\left( \CubeD[a][b][c][d][e][f][g][h][\myTikzScale] \right)
      &\;=\; \left\{ \begin{tikzpicture}[myStyle, scale=\myTikzScale]
          \CubeNodes{a}{b}{c}{}{e}{f}{g}{}
          \draw (back_nw) -- node[right=1em]{,} (front_nw)
            (back_se) -- (front_se)
            (back_sw) -- node[right=1em]{,} (front_sw);
        \end{tikzpicture} \right\}
      \;\cup\; \left\{ \begin{tikzpicture}[myStyle, scale=\myTikzScale]
          \CubeNodes{}{}{}{d}{}{}{}{h}
          \draw (back_ne) -- (front_ne);
      \end{tikzpicture} \right\} \\[1em]
    \squares_{0,1}\left( \CubeD[a][b][c][d][e][f][g][h][\myTikzScale] \right)
    &\;=\; \left\{ \begin{tikzpicture}[myStyle, scale=\myTikzScale]
        \CubeNodes{a}{b}{c}{d}{}{}{}{}
        \draw (back_nw) -- (back_ne) -- (back_se) -- (back_sw) -- (back_nw);
      \end{tikzpicture} \right\} 
    \;\cup\; \left\{ \begin{tikzpicture}[myStyle, scale=\myTikzScale]
        \CubeNodes{}{}{}{}{e}{f}{g}{h}
        \draw (front_nw) -- (front_ne) -- (front_se) -- (front_sw) -- (front_nw);
      \end{tikzpicture} \right\}
  \end{align*}
\end{subfigure}
\caption{(2)-cross section lines and (0,1)-cross section squares decomposed
  into support and pivot sets. Orientation of the labeled cube is given by
  the coordinate axes, $n = 3 = \{0,1,2\}$.}
\label{fig:lines_squares}
\end{figure}
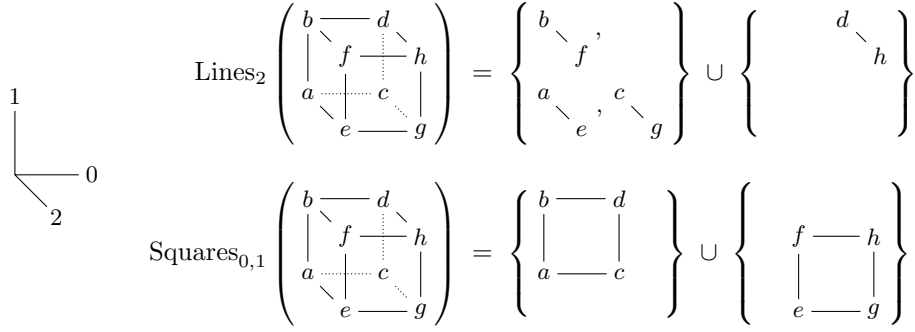   

\begin{rk}
We will often write equations in which terms are evaluated at vertex labeled
cubes which are drawn as actual cubes. This notation is a different way of
writing equations involving tuples in a product and is intended to emphasize
the geometry of the relations that are being analyzed. 
\end{rk}

Let $(\theta_0, \dots, \theta_{n-1}) \in \Con(\A)^n$ be a sequence of
congruences. The relation used to define the higher commutator is a certain
$(n)$-dimensional invariant relation that is generated by special vertex
labeled $(n)$-dimensional cubes. For each $i\in n$, let 
\[
  \gcube_i^n(x,y) \in A^{2^n}
\]
be the vertex labeled $(n)$-dimensional cube such that 
\[
  \big(\gcube_i^n(x,y) \big)_f
  = \begin{cases}
    x & \text{if } f(i) = 0, \\
    y & \text{if } f(i) = 1.
  \end{cases}
\]
Now set
\[
  \M(\theta_0, \dots, \theta_1) 
  \coloneqq \Sg_{\A^{2^n}} \left( 
    \Union_{i\in n} \Big\{\gcube_i^n(x,y) : \left< x,y \right>\in\theta_i \Big\}
  \right).
\]
This $(n)$-dimensional relation is called the algebra of \textbf{$(\theta_0,
\dots, \theta_{n-1})$-matrices}. See Figure \ref{fig:gcubes}. We can now
formulate the centrality condition used to define the higher commutator.

\begin{figure} \centering  
\def\myTikzScale{1}   
\begin{tikzpicture}[myStyle, scale=\myTikzScale]
  \path (0,0) node[yshift=-0.1em] (axes) {\CubeAxes[\theta_0][\theta_1][\theta_2][\myTikzScale]}
      ++(0.4\textwidth,0) node (cube) {$\gcube_1^3(x,y) = \CubeD[x][y][x][y][x][y][x][y][\myTikzScale]$};
  \coordinate (center) at ($(current bounding box.north west)!0.5!(current bounding box.south east)$);
  \node at ($(center)+(0,-7em)$) {$
    \M(\theta_0, \theta_1, \theta_2)
    = \Sg_{\A^{2^n}} \left\{ 
        \CubeD[a][a][b][b][a][a][b][b][\myTikzScale],
        \CubeD[c][d][c][d][c][d][c][d][\myTikzScale],
        \CubeD[e][e][e][e][f][f][f][f][\myTikzScale]
        :\ \begin{gathered}
          (a,b)\in \theta_0, \\
          (c,d)\in \theta_1, \\
          (e,f)\in \theta_2
        \end{gathered}
      \right\}
    $};
\end{tikzpicture}
\caption{Examples of a $\gcube$ and the generators of the $(3)$-dimensional
  $(\theta_0,\theta_1,\theta_2)$-matrix relation. Orientation is indicated
  by the coordinate axes. Elements of $\A$ connected by a line parallel to
  the $i$ axis are members of $\theta_i$. }
\label{fig:gcubes}
\end{figure}
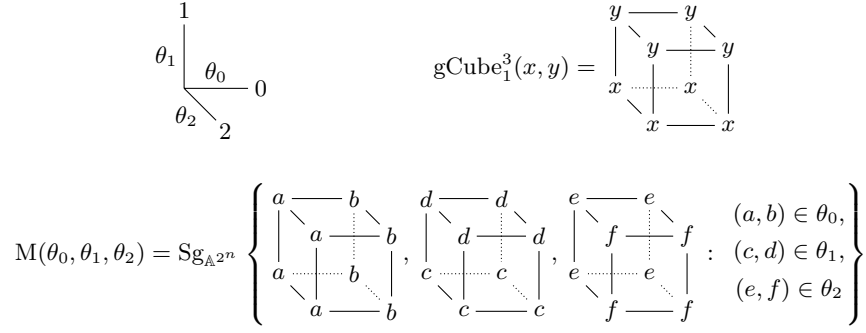   

\begin{defn}[Centrality] \label{def:centrality} 
Let $\A$ be an algebra, $2 \leq n \in \omega$, $\delta \in \Con(\A)$ and
$(\theta_0, \dots, \theta_{n-1}) \in \Con(\A)^n $ a sequence of congruences
with $\M(\theta_0, \dots, \theta_{n-1})$ defined as above. Let $\sigma \in
S_n$ be a permutation of $n$. We say that $\theta_{\sigma(0)}, \dots,
\theta_{\sigma(n-2)}$ \textbf{centralize $\theta_{\sigma(n-1)}$ modulo
$\delta$} provided the following condition holds:
\begin{quote}
  If $h \in \M(\theta_0, \dots, \theta_{n-1})$ is such that every
  $(\sigma(n-1))$-support line of $h$ is a $\delta$-pair, then the
  $(\sigma(n-1))$-pivot line of $h$ is also a $\delta$-pair.
\end{quote}
This condition is abbreviated as $\C(\theta_{\sigma(0)}, \dots,
\theta_{\sigma(n-2)}, \theta_{\sigma(n-1)}; \delta)$.
\end{defn}   

\begin{defn}[Higher Commutator] \label{def:higher_commutator} 
Under the same assumptions given in Definition \ref{def:higher_commutator},
set
\[
  \big[\theta_{\sigma(0)},\dots, \theta_{\sigma(n-1)}\big] 
  \coloneqq \Meet \big\{ \delta : \C(\theta_{\sigma(0)}, \dots, \theta_{\sigma(n-2)}, \theta_{\sigma(n-1)}; \delta) \big\}.
\]
\end{defn}   

There is some potential for confusion with this definition, because there
are several distinct algebras of matrices that can be used to define the
same commutator. Each of these algebras of matrices can be obtained from the
other by a permutation of coordinates, however. We therefore prefer, for a
given \emph{set} of congruences $\{\theta_i: i\in n\}$, to fix a coordinate
system at the outset. All centrality conditions involving the $n$-many
congruences belonging to $\{\theta_i: i \in n\}$ may then be formulated with
respect to $\M(\theta_0, \dots, \theta_{n-1})$. This is best explained
through example -- see Figures \ref{fig:centrality1} and
\ref{fig:centrality2}.

\begin{figure} \centering 
\def\myTikzScale{1.25}   
\begin{tikzpicture}[myStyle, scale=\myTikzScale]
  \path (0,0) node[yshift=-0.1em] (axes) {\CubeAxes[\theta_0][\theta_1][\theta_2][\myTikzScale]}
      ++(\textwidth*0.33,0) node (cube) {\DeltaTwoCubeD[a][b][c][d][e][f][g][h][\myTikzScale]};
\end{tikzpicture}
\caption{The condition $\C(\theta_0,\theta_1,\theta_2;\delta)$. That is, $(\theta_0,
  \theta_1)$ centralize $\theta_2$ modulo $\delta$. In Definition
  \ref{def:centrality}, $\sigma = \text{id}$.}
\label{fig:centrality1}
\end{figure}
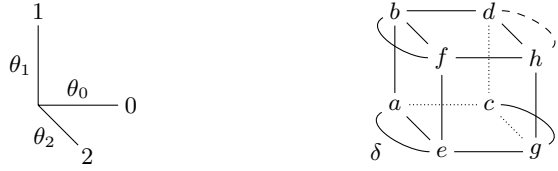   

\begin{figure} \centering 
\def\myTikzScale{1.25}   
\begin{tikzpicture}[myStyle, scale=\myTikzScale]
  \path (0,0) node (left_axes) {\CubeAxes[\theta_0][\theta_1][\theta_2][\myTikzScale]}
       +(0.4\textwidth,0) node (right_axes) {\CubeAxes[\theta_1][\theta_2][\theta_0][\myTikzScale][1][2][0]}
      ++(0.375,-2.5) node (left_cube) {\DeltaZeroCubeD[a][b][c][d][e][f][g][h][\myTikzScale]}
      +(0.4\textwidth,-0.16) node (right_cube) {\DeltaTwoCubeD[a][b][c][d][e][f][g][h][\myTikzScale]};
  \coordinate (center) at ($(current bounding box.north west)!0.5!(current bounding box.south east)$);
  \coordinate (center) at ($(left_axes)!0.5!(right_cube)$);
  \draw[->] ($(center)+(-0.75,0.25)$) to[out=12,in=180-12] node[auto]{$\sigma = (0\ 1\ 2)$} ($(center)+(0.75,0.25)$);
\end{tikzpicture}
\caption{The condition $\C(\theta_1,\theta_2,\theta_0;\delta)$. That is,
  $(\theta_1, \theta_2)$ centralize $\theta_0$ modulo $\delta$. In Definition
  \ref{def:centrality}, $\sigma = (0\ 1\ 2)$. Applying $\sigma$ to the
  coordinate axes gives a picture similar to Figure \ref{fig:centrality1}.}
\label{fig:centrality2}
\end{figure}
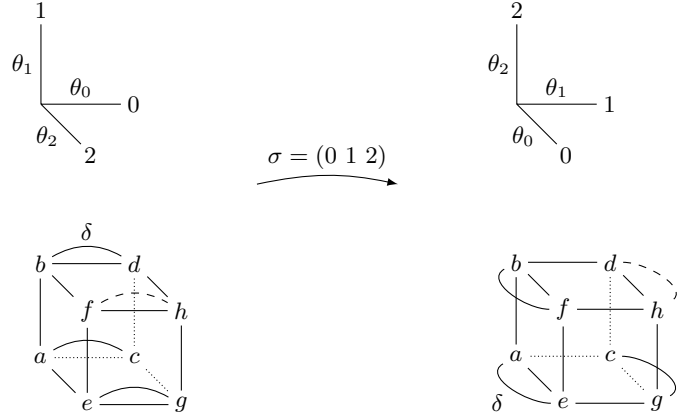   

The following properties are immediate consequences of Definition
\ref{def:higher_commutator}.
\begin{prop} \label{prop:basic_comm_properties} 
Let $\A$ be an algebra and $\alpha \in \Con(\A)$. The following hold:
\begin{enumerate}
  \item $\displaystyle{ [\alpha_0, \dots, \alpha_{k-1}] 
      \leq \Meet_{0\leq i \leq k-1}\alpha_i }$,
  \item For $\alpha_0 \leq \beta_0, \dots, \alpha_{k-1} \leq \beta_{k-1}$ in
    $\Con(\A)$, we have 
    \[
      [\alpha_0, \dots , \alpha_{k-1}] \leq [\beta_0, \dots , \beta_{k-1}].
    \]
    That is, the commutator is monotone in each argument.
  \item $\displaystyle{ 
      \underbrace{ [\alpha_0,\dots,\alpha_{k-1}] }_{k \text{-ary}} 
      \leq \underbrace{ [\alpha_1, \dots, \alpha_{k-1}] }_{(k-1) \text{-ary}} 
    }$.
\end{enumerate}
\end{prop}   

\subsection{Nilpotence, Supernilpotence, and Solvability}
Let $\A$ be an algebra and let $\alpha \in \Con(\A)$. Recursively define
over $\omega$ the congruences $[\alpha]_0 \coloneqq \alpha \eqqcolon
(\alpha]_0$,
\[
  [\alpha]_{n+1} \coloneqq \big[ [\alpha]_n, [\alpha]_n \big],
  \qquad\text{and}\qquad
  (\alpha]_{n+1} \coloneqq \big[ \alpha, (\alpha]_n \big]
\]
to produce two descending chains, called the \textbf{derived} and
\textbf{lower central series} of $\alpha$, respectively:
\[
  [\alpha_0] \geq [\alpha]_1 \geq \dots \geq [\alpha]_n \geq \dots
  \qquad\text{and}\qquad
  (\alpha_0] \geq (\alpha]_1 \geq \dots \geq (\alpha]_n \geq \dots.
\]
If $[\alpha]_n=0$ or $(\alpha]_n=0$, then $\alpha$ is said to be
\textbf{$(n)$-step solvable} or \textbf{$(n)$-step nilpotent}, respectively.
Since the binary commutator is monotonic in each of its arguments, it
follows that nilpotence is a stronger condition than solvability.

A congruence $\alpha$ of $\A$ is said to be \textbf{$(n)$-step
supernilpotent} if it satisfies 
\[
  \underbrace{[\alpha, \dots, \alpha]}_{(n-1)\text{-ary}} = 0.
\]
The reason for this terminology can be found in Aichinger and Mudrinksi
\cite{aichmud}, where it is shown that for a congruence permutable variety,
all higher commutators of appropriate arity satisfy what they call HC8,
which is an inequality involving nested commutators:
\[ \tag{HC8}
  \big[ \theta_0, \dots, \theta_{m-1},[\theta_m, \dots, \theta_{n-1}] \big] 
    \leq [\theta_0, \dots, \theta_{n-1}].
\]
Therefore, for congruence permutable varieties an easy induction shows that
if a congruence $\alpha$ is $(n)$-step supernilpotent then it must also be
$(n)$-step nilpotent (and hence also $(n)$-step solvable.)

If $\alpha = 1$ we simply say that the algebra $\A$ is $(n)$-step nilpotent,
solvable, or supernilpotent, as the case may be. We conclude this section
with a description of supernilpotence using the vocabulary that has been
developed in this paper. The proof is only a translation of definitions and
is therefore omitted. 

\begin{prop} \label{prop:supnil_supp_piv} 
Let $\A$ be an algebra, $n \geq 2$ a natural number, and $i\in n$. The
algebra $\A$ is $(n-1)$-step supernilpotent if and only if there is no
$(n)$-dimensional cube $h \in \M(1, \dots, 1)$ such that 
\begin{enumerate}
  \item every line belonging to $\SLines_k(h)$ is constant, and
  \item the $(k)$-pivot of $h$ line is not constant. 
\end{enumerate}
\end{prop}  

\section{Generalized Nilpotence and Solvability} \label{sec:generalized_nil} 
The main goal of this section is to demonstrate that the condition of
nilpotence can be quite complicated and that, for our purposes, the
condition of solvability is more useful. As noted in Section
\ref{sec:intro}, the properties of nilpotence and solvability can be defined
with the term condition commutator.

A choice was made in our definition of nilpotence to consistently evaluate
the first argument of the binary commutator at $\alpha$ and the second
argument at $(\alpha]_n$. If the commutator for $\A$ is symmetric then this
choice is immaterial, but if the commutator fails to be symmetric then this
choice is important. In the non-symmetric case, our definition of nilpotence
is demoted to what we call \textbf{left nilpotence}. The notion of
\textbf{right nilpotence} is defined in the obvious analogous way. 

Left and right nilpotence are not the same, as demonstrated by the following
example. Let $G$ and $\{o\}$ be disjoint sets with $G$ infinite. Let $A = G
\union \{o \}$ and fix some injection $s:A^2\to G$. Let $\A = \left< A ; t
\right>$ be the algebra with binary operation $t$ defined by
\[
  t(x,y)=
  \begin{cases}
    o & \text{if } x = o, \\
    s(x,y) & \text{otherwise.}
  \end{cases}
\]
$\A$ is not left nilpotent, because for each $n \in \omega$ there is a
$(1]_n$-class with infinitely many elements, namely $\{ t(a,y) : y\in A \}$
for $a\in G$ via
\[
  t\left(
    \Square[o][o][a][a][0.8],
    \Square[a][y][a][y][0.8]
  \right)
  = \SquareXY[o][o][t(a,a)][t(a,y)][1][0.8].
\]
However, $\A$ is right nilpotent. To see this, let $\delta$ be the
congruence with classes $G$ and $\{o\}$. It is a routine exercise to show
that $\C(1,1; \delta)$ holds and that $[\delta,1] = 0$. A consequence of
this is that $[1,1] \leq \delta$ and now (2) of Proposition
\ref{prop:basic_comm_properties} leads to the conclusion that  $[[1,1],1]
=0$.

A moment's reflection will reveal that the situation can be complicated. Let
$\cl{T}_{[\cdot, \cdot]}(\{x\})$ be the collection of all single-variable
terms in the binary operation symbol $[\cdot, \cdot]$. The previous
definitions of solvability and nilpotence are statements of the form
\[
  \Con(\A) \models \big( t(1)=0 \big)
\]
for some special $t(x) \in \cl{T}_{[\cdot, \cdot]}(\{x\})$, and the example
above shows that nilpotence witnessed by a particular term $t(x)$ need not
imply that all terms of a particular depth evaluate to $0$. The addition of
higher arity commutators to the language allows for more complicated terms.
That is, let
\[
  \cl{L}_n 
  = \Big\{[\cdot, \cdot], \dots, \underbrace{[\cdot, \dots, \cdot]}_{n\text{-ary}} \Big\}
\]
be the set of commutator operation symbols of arity at most $n$ and
$\cl{T}_{\cl{L}_n}(\{x\})$ be the set of all single variable terms in the
operation symbols appearing in $\cl{L}_n$. We can now ask whether 
\[
  \Con(\A) \models \big( t(1) = 0 \big) 
\]
for some $t(x) \in \cl{T}_{\cl{L}_n}(\{x\}).$

Our aim in this article is not to explore these complexities in full detail,
but rather to construct algebras $\A_n$ such that for all $t(x)\in
\cl{T}_{\cl{L}_n}(\{x\})$,
\[
  \Con(\A) \not\models \big( t(1) = 0 \big) 
\]
but
\[
  \Con(\A) \models \big( \underbrace{[1,\dots, 1]}_{(n+1)\text{-ary}} =0 \big).
\]
These two conditions say that $\A_n$ fails to be nilpotent for any
definition one could produce involving commutators up to arity $n$, but is
nevertheless $(n)$-step supernilpotent. 

We can simplify the problem by introducing a generalization of solvability.
For $n\geq 2$ and $\alpha \in \Con(\A)$, define $[\alpha]^n_0 \coloneqq
\alpha$. Now recursively define over $\omega$ the descending chain of
congruences 
\[
  [\alpha]_{m+1}^n
  \coloneqq \underbrace{ 
      \Big[ [\alpha]^n_{m}, \dots, [\alpha]^n_m \Big]
    }_{n\text{-ary}}.
\]
If $[\alpha]_m^n =0$ for some $m,n \in \omega$ we say that $\alpha$ is
\textbf{$(m)$-step solvable in dimension $n$}.

\begin{lem} \label{lem:solvable_is_lowerbound}  
Let $\A$ be an algebra, $\alpha \in \Con(\A)$, and $n \geq 2$ be a natural
number. For all $t(x) \in \cl{T}_{\cl{L}_n}(\{x\})$ there exists $m\in
\omega$ such that 
\[
  \Con(\A) \models \Big( [\alpha]_m^n \leq t(\alpha) \Big).
\]
\end{lem}
\begin{proof}
The proof proceeds by induction on the complexity of terms. It is clear that
the Lemma holds when $t(x) = x$, establishing the basis. Suppose that 
\[
  t(x) = [s_0(x), \dots, s_{k-1}(x)]
\]
for some terms $s_0, \dots, s_{k-1}$, where $k \leq n$. By the inductive
hypothesis there exist $m_0, \dots, m_{k-1}\in \omega$ such that 
\[
  \Con(\A) \models \Big( [\alpha]_{m_i}^n \leq s_i(\alpha) \Big).
\]
for each $i \in k$. Set $m$ to be the maximum of $m_0, \dots, m_{k-1}$. It
follows from (2) and (3) of Proposition \ref{prop:basic_comm_properties}
that 
\[
  \Con(\A)
  \models \bigg( \underbrace{
        \big[ [\alpha]_m^n, \dots, [\alpha]_m^n \big]
      }_{n\text{-ary}} 
    \leq \underbrace{
        \big[ [\alpha]_m^n , \dots, [\alpha]_m^n \big]
      }_{k\text{-ary}}
    \leq t(\alpha)
  \bigg).
\]
This completes the proof.
\end{proof}  

\begin{prop} \label{prop:not_solvable_is_enough}  
Let $\A$ be an algebra, $\alpha \in \Con(\A)$, and let $n \geq 2$ be a
natural number. For all $t(x) \in \cl{T}_{\cl{L}_n}(\{x\})$, 
\[
  \Con(\A) \not\models \big( t(\alpha) = 0 \big)
\]
if and only if $\alpha$ fails to be $(m)$-step solvable in dimension $n$ for
all $m\in \omega$.
\end{prop}
\begin{proof}
One direction is obvious. For the other direction, suppose that there is
some $t(x) \in \cl{T}_{\cl{L}_n}(\{x\})$ such that 
\[
  \Con(\A) \models \big( t(\alpha) = 0 \big).
\]
By Lemma \ref{lem:solvable_is_lowerbound} there is an $m \in \omega$ such
that 
\[
  \Con(\A) \models \Big( [\alpha]_m^n \leq t(\alpha) \Big).
\]
This forces $\alpha$ to be $(m)$-step solvable in dimension $n$.
\end{proof}  

\section{The Algebra $\A_2$} \label{sec:algebra_A2} 
Let $O, R, G$ be disjoint countably infinite sets where the elements of $O$
and $R$ are indexed as follows:
\[
  O = \{ o_i^j : i, j \in \omega \} 
  \qquad\text{and}\qquad
  R = \{ r_i^j: i, j \in \omega \}.
\]
Define $A_2 = O \cup R \cup G$ and let $\A_2 = \left< A_2 ; t \right>$ be
the algebra with underlying set $A_2$ and a binary operation $t: (A_2)^2 \to
A_2$ defined below.
\begin{enumerate}
  \item For all $i, j \in \omega$,
  \begin{align*}
    & t(r^j_{4i}, r^j_{4i}) = t(r^j_{4i},r^j_{4i+2}) \coloneqq o^j_i, \\
    & t(r^j_{4i+2}, r^j_{4i}) \coloneqq r^{j+1}_i, \\
    & t(r^j_{4i+2}, r^j_{4i+2}) \coloneqq r^{j+1}_{i+1}.
  \end{align*}
  This can be written compactly as
  \[
    t\left(
      \Square[r_{4i}^j][r_{4i}^j][r_{4i+2}^j][r_{4i+2}^j][1.25],
      \Square[r_{4i+2}^j][r_{4i}^j][r_{4i+2}^j][r_{4i}^j][1.25]
    \right)
    \coloneqq \Square[o_i^j][o_i^j][r_{i+1}^{j+1}][r_i^{j+1}][1.25].
  \]

  \item Otherwise, $t(x,y) \coloneqq s(x,y)$ for some injective function
  $s: (A_2)^2 \to G$.
\end{enumerate}
See Figure \ref{fig:mult_table}.

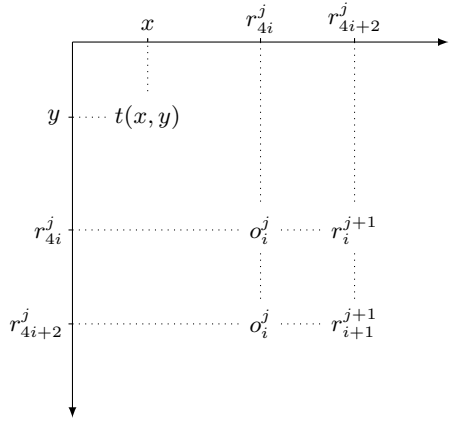
\begin{figure}  
  \begin{tikzpicture}[font=\small]
    \coordinate (origin) at (0,0);
    \draw[->] (origin) -- +(5,0);
    \draw[->] (origin) -- +(0,-5);

    \node at (1,0) [above=0.1em] (x) {$x$};
    \node at (2.5,0) [above] (x_r2i) {$r^j_{4i}$};
    \node at (3.75,0) [above] (x_r2i1) {$r^j_{4i+2}$};
    \foreach \m in {x, x_r2i, x_r2i1} {
      \coordinate (temp) at (\m|-origin);
      \draw ($(temp)+(0,0.15em)$) -- (temp);
    }

    \node at (0,-1) [left=0.1em] (y) {$y$};
    \node at (0,-2.5) [left] (y_r2i) {$r^j_{4i}$};
    \node at (0,-3.75) [left] (y_r2i1) {$r^j_{4i+2}$};
    \foreach \m in {y, y_r2i, y_r2i1} {
      \coordinate (temp) at (\m-|origin);
      \draw ($(temp)+(-0.15em,0)$) -- (temp);
    }

    \node at (x|-y) (txy) {$t(x,y)$};
    \node at (x_r2i|-y_r2i) (oji_top) {$o^j_i$};
    \node at (x_r2i1|-y_r2i) (rj1i) {$r^{j+1}_i$};
    \node at (x_r2i|-y_r2i1) (oji_bot) {$o^j_i$};
    \node at (x_r2i1|-y_r2i1) (rj1i1) {$r^{j+1}_{i+1}$};

    \draw[dotted] (x) -- (txy) -- (y)
      (x_r2i) -- (oji_top) -- (oji_bot)
      (x_r2i1) -- (rj1i) -- (rj1i1)
      (y_r2i) -- (oji_top) -- (rj1i)
      (y_r2i1) -- (oji_bot) -- (rj1i1);
  \end{tikzpicture}
  \caption{Partial Multiplication Table for $t$}
  \label{fig:mult_table} 
\end{figure}   

\subsection{$\A_2$ is Not Solvable}
We will prove that the algebra $\A_2$ fails to be $(n)$-step solvable (in
dimension $2$) for all $n \in \omega$. Recall that the derived series of
$\A_2$ is the sequence of congruences 
\[
  1 
  = [1]_0
  \geq \dots
  \geq [1]_n 
  \geq [1]_{n+1}
  \geq \dots,
\]
where $[1]_{n+1} = \big[ [1]_n, [1]_n \big].$

\begin{lem} \label{lem:infinite_class_A2}  
Let $\A_2 = \left< A_2; t\right>$ be the algebra defined at the start of
this section. For each $j \in \omega$, the set $R^j = \{ r^j_i: i \in \omega
\} \subseteq R$ is contained in a $[1]_j$-class.
\end{lem}
\begin{proof}
The proof proceeds by induction on $j$. The Lemma clearly holds for $j=0$,
establishing the basis. Suppose that $[1]_j$ has a class that contains the
set $R^j= \{r^j_i: i \in \omega\}$. It follows that 
\[
  \Square[r_{4i}^j][r_{4i}^j][r_{4i+2}^j][r_{4i+2}^j][1.25],
  \Square[r_{4i+2)}^j][r_{4i}^j][r_{4i+2}^j][r_{4i}^j][1.25]
  \in \M\big( [1]_j, [1]_j \big)
\]
for each $i\in \omega$. Therefore, 
\[
  t\left(
    \Square[r_{4i}^j][r_{4i}^j][r_{4i+2}^j][r_{4i+2}^j][1.25],
    \Square[r_{4i+2}^j][r_{4i}^j][r_{4i+2}^j][r_{4i}^j][1.25]
  \right)
  = \Square[o_i^j][o_i^j][r_{i+1}^{j+1}][r_i^{j+1}][1.25]
  \in \M([1]_j, [1]_j).
\]
Since $\big< o^j_i, o^j_i \big> \in \big[ [1]_j,[1]_j \big]$, we conclude
that $\big< r^{j+1}_i, r^{j+1}_{i+1} \big> \in [1]_{j+1}$ for each $i\in
\omega$. Equivalence relations are transitively closed, so it follows that
$\big< r_0^{j+1}, r_i^{j+1} \big> \in [1]_{j+1}$ for each $i\in \omega$.
Therefore, $R^{j+1} = \{r_i^{j+1}: i \in \omega \}$ is a subset of the class
of $[1]_{j+1}$ that is represented by $r_0^{j+1}$. This completes the
induction and the proof.
\end{proof}  

\begin{thm} \label{thm:not_solvable_A2} 
The algebra $\A_2 = \left< A_2; t \right>$ is not solvable (in dimension 2). 
\end{thm}
\begin{proof}
If $\A_2$ were solvable then there would exist an $n\in \omega$ such that 
\[
  [1]_n = 0.
\]
In particular, every class of $[1]_n$ would contain exactly one element, but
Lemma \ref{lem:infinite_class_A2} ensures the existence of a class with
infinitely many elements.
\end{proof}  

\subsection{$\A_2$ is Supernilpotent} 
We will now prove that the algebra $\A_2$ is $(2)$-step supernilpotent. The
proof is an induction on the complexity of terms that generate the algebra
of $(1,1,1)$-matrices (i.e.\ $\M(1,1,1)$). Before embarking on the proof,
however, we must build up some of the necessary machinery. The following
lemmas are proved in full generality in Section \ref{sec:algebra_An} at the
start of Subsection \ref{subsec:An_supernilp}. There is not a strong
geometrical intuition that can be gained from examining the lower-dimension
proofs, so we refer the reader to the next section for detailed
justification of these lemmas.

\begin{lem} \label{lem:trivial_A2} 
Let $\A_2 = \left< A_2 ; t\right>$ be the algebra defined at the start of
this section.
\begin{enumerate}
  \item If $t(a,b) \in R \cup O$ then $a,b\in R$.

  \item If $t(a,b) = t(c,d)\not\in O$ then $(a,b) = (c,d)$.

  \item If $t(a,b) = t(c,d)$ and $(a,b) \neq (c,d)$, then
  \begin{itemize}
    \item $t(a,b) = t(c,d) = o_i^j$ for some $o_i^j\in O$,
    \item $a = c = r_{4i}^j$, and 
    \item $\big\{ b,d \big\} = \big\{ r^j_{4i}, r^j_{4i+2} \big\}$.
  \end{itemize}
\end{enumerate}
\end{lem}  

\begin{lem} \label{lem:must_be_sucessors_A2} 
If
\[
  h
  = \Square[c][c][r_k^\ell][r_i^j]
  \in \M(1,1)
  \quad\text{for some }
  c\in A_2,
\]
then $j = \ell$ and $|i-k| \in \{0,1\}$.
\end{lem}  

\begin{lem} \label{lem:must_be_generators_A2}  
If
\[
  h
  = \Square[r_u^v][r_i^j][a][r_\ell^k]
  \in \M(1,1)
\]
for some $r_i^j, r_l^k, r_u^v \in R$ and $a \in A_2$, then
\[
  h
  \in \left\{ \Square[y][x][y][x], \Square[x][x][y][y] : x,y \in A_2 \right\}.
\]
\end{lem}  

We are now ready to prove that $\A_2$ is not $(2)$-step supernilpotent.
Although the proof of this theorem can be worked out from the proof of the
higher-dimensional Theorem \ref{thm:An_supernilp}, we include it here in the
hope that it will provide some geometrical intuition for the general case.

\begin{thm} \label{thm:A2_supernilp} 
The algebra $\A_2 = \left< A_2; t \right> $ is $(2)$-step supernilpotent. 
\end{thm}
\begin{proof}
By Proposition \ref{prop:supnil_supp_piv}, $\A_2$ is $(2)$-step supernilpotent
if and only if
\[
  h = \CubeD[a][a][b][b][c][c][d][e] \in \M(1,1,1)
\]
implies $e = d$. In other words, if the vertical support lines are constant,
then the vertical pivot line is constant as well. Set 
\begin{align*}
  X_0 
  &= \left\{ 
      \CubeD[x][x][y][y][x][x][y][y], 
      \CubeD[x][y][x][y][x][y][x][y], 
      \CubeD[x][x][x][x][y][y][y][y]: x,y \in A_2\right\}
  &\text{and} \\
  X_{n+1} 
    &= X_n \cup \Big\{ t(a,b): a, b \in X_n \Big\}.
\end{align*}
By definition, $\M(1,1,1) = \Sg_{(A_2)^{2^3}}(X) = \Union_{n\in \nat} X_n$.
We proceed by induction on $n$. 

For a cube $h\in X_0$ it is true that having constant vertical support lines
implies a constant vertical pivot line, establishing the basis. Suppose now
that this implication holds for $X_n$ and that 
\[
  h = \CubeD[a][a][b][b][c][c][d][e] \in X_{n+1}\setminus X_n.
\]
We will show that $d = e$. We have that
\[
  h 
  = \CubeD[a][a][b][b][c][c][d][e]
  = t\left( 
    \CubeD[a_0][a_0'][b_0][b_0'][c_0][c_0'][d_0][e_0],
    \CubeD[a_1][a_1'][b_1][b_1'][c_1][c_1'][d_1][e_1]
  \right),
\]
where the two argument cubes are elements of $X_n$. From Lemma
\ref{lem:trivial_A2}, it must be that $a_0 = a_0'$, $b_0 = b_0'$, and $c_0 =
c_0'$. Applying the inductive hypothesis to the first argument cube now
yields $e_0 = d_0$, so the first argument cube has its bottom face equal to
its top face. Observe that we need only prove that $e_1 = d_1$ since $e_0 =
d_0$ already. The situation is now
\begin{equation}  \label{eqn:thm_A2_supernilp}
  h 
  = \CubeD[a][a][b][b][c][c][d][e]
  = t\left( 
    \CubeD[a_0][a_0][b_0][b_0][c_0][c_0][d_0][d_0],
    \CubeD[a_1][a_1'][b_1][b_1'][c_1][c_1'][d_1][e_1]
  \right).
\end{equation}
Let $S$ be the set of vertical support lines of the second argument cube and
let $D$ be the set of constant lines:
\[
  S = \Bigg\{ 
        \LinePic[a_1][a_1'][0.75],
        \LinePic[b_1][b_1'][0.75],
        \LinePic[c_1][c_1'][0.75]
      \Bigg\},
  \qquad\qquad
  D = \Bigg\{ \LinePic[\alpha][\alpha][0.75] : \alpha\in A_2 \Bigg\}.
\]
We will proceed with a case analysis of $S\cap D$.

\Case{$S\cap D \neq \emptyset$} 
In this case, there is a constant vertical support line of the second
argument cube, say $c_1 = c_1'$. Suppose towards a contradiction that
$a_1\neq a_1'$. Since $h$ has all constant vertical support lines, $t$
evaluated at this line must be a failure of injectivity. By Lemma
\ref{lem:trivial_A2}, it must be that (modulo vertical reflection) $a =
o_i^j$, $a_0 = r_{4i}^j$, $a_1 = r_{4i}^j$, and $a_1' = r_{4i+2}^j$.
Equation \eqref{eqn:thm_A2_supernilp} is now
\[
  h 
  = \CubeD[o_i^j][o_i^j][b][b][c][c][d][e]
  = t\left( 
    \CubeD[r_{4i}^j][r_{4i}^j][b_0][b_0][c_0][c_0][d_0][d_0],
    \CubeD[r_{4i}^j][r_{4i+2}^j][b_1][b_1'][c_1][c_1][d_1][e_1]
  \right).
\]
If we apply Lemma \ref{lem:must_be_sucessors_A2} to the left face of the
second cube, we obtain a contradiction, since $|4i+2-4i| = 2 \not\in
\{0,1\}$. It follows that $a_1 = a_1'$. We can continue around the cube in
this manner to obtain all constant vertical support lines, forcing $e_1 =
d_1$ by the inductive hypothesis. This, in turn, implies that $e = d$. At
the start of this case we assumed that the $(c_1, c_1')$ line was the
constant vertical support line, but the above argument works no matter which
support line is constant.

\Case{$S\cap D = \emptyset$}
In this case, from the definition of $t$ and Lemma \ref{lem:trivial_A2}, 
equation \ref{eqn:thm_A2_supernilp} looks like
\[
  h
  = \CubeD[o_k^{\ell}][o_k^{\ell}][o_m^n][o_m^n][o_i^j][o_i^j][d][e][1.25]
  = t\left( 
    \CubeD[r^\ell_{4k}][r^\ell_{4k}][r^n_{4m}][r^n_{4m}][r^j_{4i}][r^j_{4i}][d_0][d_0][1.25],
    \CubeD[r^\ell_{4k+\alpha}][r^\ell_{4k+\beta}][r^n_{4m+\gamma}][r^n_{4m+\delta}][r^j_{4i+\epsilon}][r^j_{4i+\tau}][d_1][e_1][1.25]
  \right),
\]
where $\{0,2\} = \{\alpha,\beta\} = \{\gamma, \delta\} = \{\epsilon,
\tau\}$. Applying Lemma \ref{lem:must_be_generators_A2} to the leftmost
face, we obtain $r_{4k+\alpha}^\ell = r_{4i+\epsilon}^j$ and
$r_{4k+\beta}^\ell = r_{4i+\tau}^j$. Similarly, the back face implies that
$r_{4k+\alpha}^\ell = r_{4m+\gamma}^n$ and $r_{4k+\beta}^\ell =
r_{4m+\delta}^n$. The situation is now (modulo vertically flipping the
second cube)
\[
  h
  = \CubeD[o_k^{\ell}][o_k^{\ell}][o_k^\ell][o_k^\ell][o_k^\ell][o_k^\ell][d][e][1.25]
  = t\left( 
    \CubeD[r^\ell_{4k}][r^\ell_{4k}][r^\ell_{4k}][r^\ell_{4k}][r^\ell_{4k}][r^\ell_{4k}][d_0][d_0][1.25],
    \CubeD[r^\ell_{4k}][r^\ell_{4k+2}][r^\ell_{4k}][r^\ell_{4k+2}][r^\ell_{4k}][r^\ell_{4k+2}][d_1][e_1][1.25]
  \right).
\]
Applying Lemma \ref{lem:must_be_generators_A2} to the top and bottom faces
of the first argument cube implies that $d_0 = r^\ell_{4k}$. Applying Lemma
\ref{lem:must_be_generators_A2} to the top and bottom faces of the second
argument cube implies that $d_1 = r^\ell_{4k}$ and $e_1 = r^\ell_{4k+2}$.
Evaluating it all gives us $e = o_k^\ell = d$, as desired.

This completes the case analysis and the induction.
\end{proof}  

\section{The Algebra $\A_n$} \label{sec:algebra_An} 
Let $n\in \omega$. Let $O, R, G$ be disjoint sets, indexed as follows
\[
  O = \big\{ o^j_{i,g} : i,j\in \omega,\ g\in 2^{n-1} \big\}
  \qquad\text{and}\qquad
  R = \big\{ r^j_i : i,j\in \omega \big\}.
\]
Define $A_n = O\cup R\cup G$ and let $\A_n = \left< A_n; t\right>$ be the
algebra with underlying set $A_n$ and an $n$-ary operation $t:(A_n)^n\to
A_n$ with values given by the cube equation
\begin{multline} \label{eqn:An_t_defn}
  \Bigg( t\Big(
    \gcube_0^n(r_{4i}^j, r_{4i+2}^j),
    \dots,
    \gcube_{n-1}^n(r_{4i}^j, r_{4i+2}^j) 
  \Big) \Bigg)_f \\
  \coloneqq \begin{cases}
    r_i^{j+1}     &\text{if } f = (1,\dots,1,0)\in 2^n, \\
    r_{i+1}^{j+1} &\text{if } f = (1,\dots,1,1)\in 2^n, \\
    o^j_{i,g}     &\text{if } f|_{n-1} = g \neq \textbf{1} \in 2^{n\setminus \{n-1\}},
  \end{cases}
\end{multline}
and otherwise $t$ takes the value of some fixed injection $s:(A_n)^n\to G$.
See Figure \ref{fig:An_t}.

\begin{figure}[!ht] \centering 
\def\myTikzScale{1.28}   
\begin{tikzpicture}[myStyle, scale=\myTikzScale]
  \path (-2.25,0) node (3axes) {\CubeAxes[][][][1.35]} 
        (2.25,0) node (3coords) {\CubeD[(0,0,0)][(0,1,0)][(1,0,0)][(1,1,0)][(0,0,1)][(0,1,1)][(1,0,1)][(1,1,1)][1.35]};
  \node (3t_def) at (0,-2.25) {$
      t\left( 
        \CubeD[r_{4i}^j][r_{4i}^j][r_{4i+2}^j][r_{4i+2}^j][r_{4i}^j][r_{4i}^j][r_{4i+2}^j][r_{4i+2}^j][\myTikzScale],
        \CubeD[r_{4i}^j][r_{4i+2}^j][r_{4i}^j][r_{4i+2}^j][r_{4i}^j][r_{4i+2}^j][r_{4i}^j][r_{4i+2}^j][\myTikzScale],
        \CubeD[r_{4i}^j][r_{4i}^j][r_{4i}^j][r_{4i}^j][r_{4i+2}^j][r_{4i+2}^j][r_{4i+2}^j][r_{4i+2}^j][\myTikzScale]
      \right)
      = \CubeD[o_{i,(0,0)}^j][o_{i,(0,1)}^j][o_{i,(1,0)}^j][r_i^{j+1}][o_{i,(0,0)}^j][o_{i,(0,1)}^j][o_{i,(1,0)}^j][r_{i+1}^{j+1}][\myTikzScale]
    $};
\end{tikzpicture}
\caption{Equation \eqref{eqn:An_t_defn} with associated $(n)$-dimensional
  cube coordinate system for $n=3$.}
\label{fig:An_t}
\end{figure}
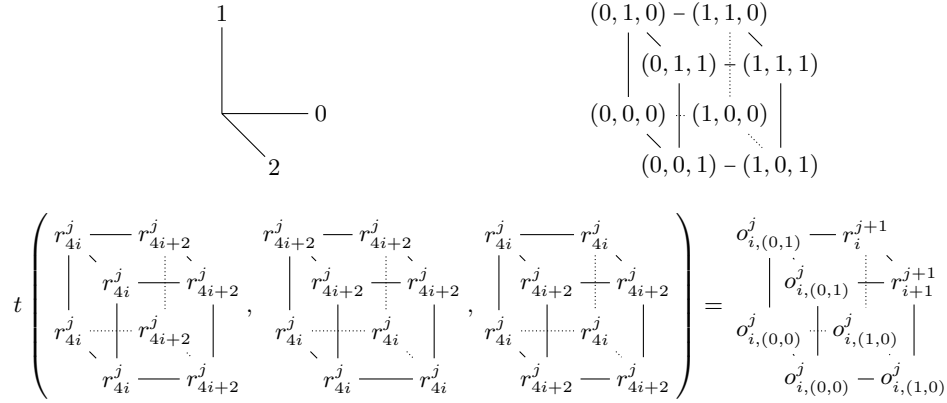   

\subsection{$\A_n$ is Not Solvable in Dimension $n$}
In this section we prove that the algebra $\A_n$ fails to be $(j)$-step
solvable in dimension $n$ for all $j \in \omega$. Recall that the
$(n)$-dimensional generalization of the derived series of $\A_n$ is the
sequence of congruences 
\[
  1 
  = [1]_0^n 
  \geq \dots 
  \geq [1]_j^n 
  \geq [1]_{j+1}^n 
  \geq \dots,
\]
where $[1]_{j+1}^n = \big[ [1]_j^n, \dots, [1]_j^n \big]$ ($n$-ary). We now
repeat the same analysis that we did in the previous section.

\begin{lem} \label{lem:infinite_class_An}  
Let $\A_n = \left< A_n; t\right>$ be the algebra defined above. For each $j
\in \omega$, the set $R^j = \{r^j_i: i \in \omega\} \subseteq R$ is
contained in a $[1]_j^n$-class.
\end{lem}
\begin{proof}
The proof proceeds by induction on $j$. The Lemma clearly holds for $j=0$,
establishing the basis. Suppose that $[1]_j^n$ has a class that contains the
set $R^j= \{r^j_i: i \in \omega\}$. It follows that 
\[
  \left\{ \gcube_k^n(r^j_{4i}, r^j_{4i+2}): k\in n \right\}
    \subseteq \M\big( \underbrace{[1]^n_j,\dots,  [1]^n_j}_{n} \big)
\]
for each $i\in \omega$. Therefore, 
\[
  h
  = t\Big(
      \gcube_0^n(r_{4i}^j, r_{4i+2}^j),
      \dots,
      \gcube_{n-1}^n(r_{4i}^j, r_{4i+2}^j) 
    \Big)
  \in \M\big( [1]^n_j,\dots,  [1]^n_j \big).
\]
By the definition of $t$, every element of $\SLines_{n-1}(h)$ is constant.
It follows that the pair determined by the $(n-1)$-pivot line of $h$ must
belong to $\big[[1]^n_j, \dots, [1]^n \big]$ ($n$-ary). The pair determined
by the $(n-1)$-pivot line of $h$ is $\big< r_i^{j+1}, r_{i+1}^{j+1} \big>$,
so we have shown that
\[
  \big< r_i^{j+1}, r_{i+1}^{j+1} \big> \in [1]_{j+1}^n.
\]
It follows that $\big< r_0^{j+1}, r_i^{j+1} \big> \in [1]_{j+1}$ for each
$i\in \omega$. Therefore, $R^{j+1} = \{r_i^{j+1}: i \in \omega \}$ is a
subset of the class of $[1]^n_{j+1}$ that is represented by $r_0^{j+1}$.
This completes the induction and the proof.
\end{proof}  

\begin{thm} \label{thm:not_solvable_An} 
The algebra $\A_n = \left< A_n; t \right>$ is not solvable (in dimension n). 
\end{thm}
\begin{proof}
If $\A_n$ were solvable in dimension $n$ then there would exist an $m\in
\omega$ such that 
\[
  [1]_m^n= 0.
\]
In particular, every class of $[1]_m^n$ would contain exactly one element,
but Lemma \ref{lem:infinite_class_An} ensures the existence of a class with
infinitely many elements.
\end{proof}  

\subsection{$\A_n$ is $(n)$-step Supernilpotent} \label{subsec:An_supernilp}
We now prove versions of Lemmas \ref{lem:trivial_A2},
\ref{lem:must_be_sucessors_A2}, and \ref{lem:must_be_generators_A2} for $\A_n$.
These lemmas describe in detail the exact manner in which the operation $t$
fails to be injective, and the different kinds of squares that can appear in
$\M(1,1)$. The following Lemma follows immediately from the definition of
$\A_n$ and is therefore omitted.

\begin{lem} \label{lem:trivial_An} 
Let $\A_n = \left< A_n ; t\right>$ be the algebra defined at the start of
this section.
\begin{enumerate}
  \item If $t(\overline{a}) \in R \cup O$ then $\overline{a}\in R^n$.

  \item If $t(\overline{a}) = t(\overline{b})\not\in O$ then $\overline{a} =
  \overline{b}$.

  \item If $t(\overline{a}) = t(\overline{b})$ and $\overline{a} \neq
  \overline{b}$, then
  \begin{itemize}
    \item $t(\overline{a}) = t(\overline{b}) = o_{i,g}^j$ for some
      $o_{i,g}^j\in O$,
    \item $a_k = b_k\in \big\{ r_{4i}^j, r_{4i+2}^j \big\}$ for all 
      $k\in (n-1)$, and
    \item $\big\{ a_{n-1}, b_{n-1} \big\} = \big\{ r^j_{4i}, r^j_{4i+2} \big\}$.
  \end{itemize}
\end{enumerate}
\end{lem}  

We are now ready to begin our analysis of the squares in $\M(1,1)$.

\begin{lem} \label{lem:must_be_sucessors_An} 
If
\[
  h
  = \Square[c][c][r_k^\ell][r_i^j]
  \in \M(1,1)
  \quad\text{for some }
  c\in A_n,
\]
then $j = \ell$ and $|i-k| \in \{0,1\}$.
\end{lem} 
\begin{proof}
The proof shall be by induction on the level at which $h$ first appears
during subalgebra generation. Let
\begin{align*}
  X_0
    &= \left\{ \Square[y][x][y][x], \Square[x][x][y][y] : x,y \in A_n \right\} 
    &\text{and} \\
  X_{m+1} 
    &= X_m \cup \Big\{ t(\overline{a}): \overline{a}\in (X_m)^n \Big\}.
\end{align*}
By definition, $\M(1,1) = \Sg_{(\A_n)^{2^2} }(X_0) = \bigcup_{m\in\omega}
X_m$. We will proceed by induction on $m$. The Lemma clearly holds for $h\in
X_0$, establishing the basis. Suppose now that the Lemma holds for $X_m$ and
that 
\[
  h
  = \Square[c][c][r_k^\ell][r_i^j] 
  \in X_{m+1}\setminus X_m.
\]
We will prove that $j = \ell$ and $|i-k|\in \{0,1\}$.

From Lemma \ref{lem:trivial_An}, it must be that
\[
  h 
  = \Square[c][c][r_k^\ell][r_i^j]
  = t\Bigg( \underbrace{
    \Square[a_0][a_0][r_{v_0}^{\ell_0-1}][r_{u_0}^{j_0-1}],
    \dots,
    \SquareXY[a_{n-2}][a_{n-2}][r_{v_{n-2}}^{\ell_{n-2}-1}][r_{u_{n-2}}^{j_{n-2}-1}][1.35],
    \Square[a_{n-1}][b_{n-1}][\alpha][\beta]
  }_{\in X_m} \Bigg)
\]
(note that the last square need not have equal vertical lines). Applying the
inductive hypothesis to the first $(n-1)$ argument squares gives us $j_p =
\ell_p$ and $|u_p-v_p|\in \{0,1\}$ for all $p\in (n-1)$.

Consider the evaluation of $t$ on the rightmost vertical lines of the
argument squares:
\[
  t\Bigg(
    \LinePic[r_{v_0}^{\ell_0-1}][r_{u_0}^{j_0-1}],
    \dots,
    \LinePic[r_{v_{n-2}}^{\ell_{n-2}-1}][r_{u_{n-2}}^{j_{n-2}-1}],
    \LinePic[\alpha][\beta]
  \Bigg)
  = \LinePic[r_k^\ell][r_i^j].
\]
From Lemma \ref{lem:trivial_An} and the definition of $t$, the only way that
this is possible is if there are some $\epsilon,\tau\in\{0,1\}$ such that
for all $p\in (n-1)$
\begin{align*}
  j_p &= j,        & u_p &= 4(i-\epsilon)+2,  & \alpha &= r_{4(i-\epsilon)+2\epsilon}^j, \\
  \ell_p &= \ell,  & v_p &= 4(k-\tau)+2,      & \beta &= r_{4(k-\tau)+2\tau}^\ell.
\end{align*}
The reader is encouraged to consult Figure \ref{fig:An_t}. Combining this
with the conclusions from the end of the previous paragraph, we have that $j
= \ell$ and 
\[
  | u_p - v_p |
  = \Big| \big( 4(i-\epsilon)+2 \big) - \big( 4(k-\tau)+2 \big) \Big|
  = 4 \big| (i-k) - (\epsilon-\tau) \big|
  \in \{0,1\}.
\]
This implies that $i-k = \epsilon-\tau$. For all possibilities of
$\epsilon,\tau\in \{0,1\}$, we have $|i - k|\in \{0,1\}$. This completes the
induction, and finishes the proof.
\end{proof} 

\begin{lem} \label{lem:must_be_generators_An}  
If
\[
  h
  = \Square[r_u^v][r_i^j][a][r_\ell^k]
  \in \M(1,1)
\]
for some $r_i^j, r_l^k, r_u^v \in R $ and $a \in A_n$, then
\[
  h
  \in \left\{ \Square[y][x][y][x], \Square[x][x][y][y] : x,y \in A_n \right\}.
\]
\end{lem}
\begin{proof}
The proof is similar to the proof of Lemma \ref{lem:must_be_sucessors_An},
and we begin the same way. Let
\begin{align*}
  X_0
    &= \left\{ \Square[y][x][y][x], \Square[x][x][y][y] : x,y \in A_n \right\} 
    &\text{and} \\
  X_{m+1} 
    &= X_m \cup \Big\{ t(\overline{a}): \overline{a}\in (X_m)^n \Big\}.
\end{align*}
so that $\M(1,1) = \Sg_{(\A_n)^{2^2}}(X_0) = \Union_{m\in \omega} X_m$.
We proceed by induction on $m$. The Lemma trivially holds for $h\in X_0$,
establishing the basis. Suppose now that the Lemma holds for $X_n$ and that 
\[
  h
  = \Square[r_u^v][r_i^j][a][r_\ell^k] 
  \in X_{m+1} \setminus X_m.
\]
We will prove that $h\in X_0$. As in Lemma \ref{lem:must_be_sucessors_An},
this implies (from the definition of $t$ and by Lemma \ref{lem:trivial_An})
that
\begin{align*}
  h
  &= \Square[r_u^v][r_i^j][a][r_\ell^k] \\
  &= t\Bigg( \underbrace{
      \SquareXY[r_{4(u-\epsilon)+2}^{v-1}][r_{4(i-\tau)+2}^{j-1}][a_0][r_{4(\ell-\sigma)+2}^{k-1}][1.75],
      \dots,
      \SquareXY[r_{4(u-\epsilon)+2}^{v-1}][r_{4(i-\tau)+2}^{j-1}][a_{n-2}][r_{4(\ell-\sigma)+2}^{k-1}][1.75],
      \SquareXY[r_{4(u-\epsilon)+2\epsilon}^{v-1}][r_{4(i-\tau)+2\tau}^{j-1}][a_{n-1}][r_{4(\ell-\sigma)+2\sigma}^{k-1}][1.9]
    }_{\in X_m} \Bigg)
\end{align*}
for some $\epsilon, \tau, \sigma\in \{0,1\}$. The reader is encouraged to
consult Figure \ref{fig:An_t}. The inductive hypothesis applies to each of
the argument squares, so for each square the columns are constant or the
rows are constant. 

By symmetry, we may assume without loss of generality that the last argument
square has constant columns. This implies that $j = v$ and that
$4(i-\tau)+2\tau = 4(u-\epsilon)+2\epsilon$. This last equation reduces to
$2i-\tau = 2u-\epsilon$. Since $\epsilon,\tau\in \{0,1\}$, we have that
$\tau = \epsilon$ and thus $i = u$. This forces the first column of all the
argument squares to be constant, which in turn (by the inductive hypothesis)
forces the second columns of all the argument squares to be constant. Hence
$h$ has constant columns, and so $h\in X_0$, completing the induction.
\end{proof} 

In the previous section, the above $n=2$ version of Lemma
\ref{lem:must_be_generators_An} above was sufficient to analyze the cubes in
Theorem \ref{thm:A2_supernilp} since the faces of $(3)$-dimensional cubes
are squares. The faces of $(n+1)$-dimensional cubes, however, are
$(n)$-dimensional cubes. The analysis which must be performed is therefore
aided by generalizing the above lemma to $(n)$-dimensional cubes rather than
squares.

\begin{lem} \label{lem:must_be_generators_ultimate} 
Let
\[
  h \in \M(\underbrace{1, \dots, 1}_{n\geq 2})
\]
be an $(n)$-dimensional cube for $n\geq 2$. If we have $\SLines_{n-1}(h)
\subseteq R^2$ then $h = \gcube_i^n(r', r'') $ for some $i \in n$ and $r',
r'' \in R$. 
\end{lem}
\begin{proof}
Observe that when $n = 2$, this is just Lemma
\ref{lem:must_be_generators_An}. We first show that that $h \in R^{2^n}$.
Since $\SLines_{n-1}(h)\subseteq R^2$, we need only show that the
$(n-1)$-pivot line of $h$ lies in $R^2$. The two vertices of this line are
$h_{\textbf{1}}$ and $h_f$ where $\textbf{1}\in 2^n$ and $f =
(1,\dots,1,0)\in 2^n$. The $(0,1)$-pivot square and $(0,n-1)$-pivot square
of $h$ are
\[
  \underbrace{ \Square[r'][r''][r'''][h_\textbf{1}] }_{\text{$(0,1)$-pivot}}
  \qquad \text{and} \qquad
  \underbrace{ \Square[r'][r''][h_f][h_\textbf{1}] }_{\text{$(0,n-1)$-pivot}}.
\]
Applying Lemma \ref{lem:must_be_generators_An} to the first square and then
the second yields $h_\textbf{1}, h_f\in R$, proving that $h\in R^{2^n}$.
Lemma \ref{lem:must_be_generators_An} applied to all the cross section
squares of $h$ proves that each cross section square must be of the form
$\gcube_i^2(r',r'')$ for some $i \in 2$ and $r', r'' \in R$. The proof will
be finished after we establish the following claim.

\begin{claim}
Let $m \geq 2$ be an integer, $S$ a set, and $h \in S^{2^m}$ an
$(m)$-dimensional cube. If every cross section square is of the form
$\gcube_i^2(a,b)$ for some $i \in 2$ and $a,b \in S$ then $h =
\gcube^m_j(c,d)$ for some $j\in n$ and $c,d \in S$.
\end{claim}
\begin{claimproof}
We proceed by induction on the dimension $m$. The claim is trivial if $m=2$.
Assume now that it holds for $m \geq 2$ and take $h\in S^{2^{m+1}}$
satisfying the hypotheses of the claim. Denote by $(m\mapsto 0)$ and
$(m\mapsto 1)$ the functions from the singleton $\{m\}$ into $2$ that assign
$m$ the value $0$ and $1$, respectively. The inductive assumption implies
that 
\[
  h_{(m\mapsto 0)} = \gcube_{j_0}^m(a_0, b_0)
  \qquad \text{and} \qquad
  h_{(m\mapsto 1)}=\gcube_{j_1}^m(a_1, b_1)
\]
for some $j_0, j_1 \in m$ and $a_0,b_1, a_1, b_1 \in S$. If both $a_0=b_0$
and $a_1=b_1$ then $h = \gcube_m^{m+1}(a_0, a_1)$. We therefore assume that
$a_0 \neq b_0$. A typical $(m, j_0)$-cross section square of $h$ looks like
\[
  \Square[a_0][b_0][c][d],
\]
where $\left< a_0, b_0\right>$ and $\left< c, d \right>$ are $(j_0)$-cross
section lines of $h_{(m\mapsto 0)}$ and $h_{(m\mapsto 1)}$, respectively. By
hypothesis, this square must be of the form $\gcube_i^2(a_0',b_0')$. Since
$a_0\neq b_0$, it must be that $i = 1$, $a_0' = a_0 = c$, and $b_0' = b_0 =
d$. Applying the same argument to a $(m,j_1)$-cross section square yields
$a_0 = a_1$ and $b_0 = b_1$. In turn, this now implies $j_1 = j_0$. Putting
it all together, we have $h = \gcube_{j_0}^{m+1}(a_0,b_0)$, proving the
claim.
\end{claimproof}
\end{proof}  

We are now ready to prove the general version of Theorem
\ref{thm:A2_supernilp}.
\begin{thm} \label{thm:An_supernilp} 
The algebra $\A_n = \left< A_n; t \right>$ is $(n)$-step supernilpotent. 
\end{thm}
\begin{proof}
By Proposition \ref{prop:supnil_supp_piv}, we must show that for all
$(n+1)$-dimensional cubes
\[
  h\in \M(\underbrace{1,\ldots,1}_{n+1}),
\]
if $\SLines_n(h)$ has all constant edges, then the $(n)$-pivot line is
constant as well. Let
\begin{align*}
  X_0
    &= \Big\{ \gcube_i^{n+1}(x,y) : x,y\in A_n,\ i\in n+1 \Big\} \text{ and} \\
  X_{m+1}
    &= X_m \cup \Big\{ t(\overline{a}): \overline{a} \in (X_m)^n \Big\}.
\end{align*}
Note that $\M(1,\ldots,1) = \Sg_{(\A_n)^{2^{n+1}} }(X_0) = \bigcup_{m\in\omega}
X_m$. We will proceed by induction on $m$. 

For $h\in X_0$ it is true that having constant $(n)$-support lines implies
having a constant $(n)$-pivot line, establishing the basis. Suppose now that
this implication holds for $X_m$ and that 
\[
  h \in X_{m+1}\setminus X_m
\]
has $\SLines_n(h)$ constant. We will show that the $(n)$-pivot line must
also be constant. Since $h\in X_{m+1}\setminus X_m$, there are cubes
$c_0,\ldots,c_{n-1}\in X_m$ such that
\[
  h = t\big(c_0,\ldots, c_{n-2}, c_{n-1} \big).
\]
Now, the $(n)$-support line of $h_f$ for a particular $f \in
2^{(n+1)\setminus \{n\}}$ is of the form
\[
  h_f = 
  t \Bigg( \LinePic[b_0][a_0][0.75], \dots, \LinePic[b_{n-1}][a_{n-1}][0.75] \Bigg),
\]
where for each $d\in n$, $\big< a_d, b_d \big>$ is the $(n)$-support line of
$(c_d)_f$. Lemma \ref{lem:trivial_An} implies that $a_d = b_d$ for all $d
\in (n-1)$ and either $a_{n-1} = b_{n-1}$ or $\{ a_{n-1}, b_{n-1} \} = \{
r^j_{4i}, r^j_{4i+2} \}$ for some $i,j \in \omega$. The inductive hypothesis
applied to $c_d$ for $d \in (n-1)$ implies that the $(n)$-pivot line of
$c_d$ (that is, $(c_d)_{\textbf{1}}$) is constant. Succinctly, we have
determined that
\begin{align*}
 \lines_n(c_d)
   &\subseteq \Bigg\{ \LinePic[c][c][0.75] : c\in A_n \Bigg\} 
     \text{ for all $d\in (n-1)$ and} \\
 \SLines_n(c_{n-1})
  &\subseteq \Bigg\{ \LinePic[c][c][0.75] : c\in A_n \Bigg\}
    \union \Bigg\{ \LinePic[r^j_{4i+\epsilon}][r^j_{4i+\tau}][0.75] 
    : i,j\in \omega,\ \{\epsilon, \tau\} = \{0,2\} \Bigg\}.
\end{align*}
Observe that if the $(n)$-pivot line of $c_{n-1}$ is constant then the
$(n)$-pivot line of $h$ will be constant as well. Let $D$ be the set of
constant lines:
\[
  D = \Bigg\{ \LinePic[c][c][0.75] : c\in A_n \Bigg\}.
\]
We now proceed with a case analysis of $\SLines_n(c_{n-1})\cap D$.

\Case{$\SLines_n(c_{n-1}) \cap D \neq \emptyset$}
In this case, there is some constant $(n)$-support line of $c_{n-1}$. This
is enough to force every $(n)$-support line of $c_{n-1}$ to be constant. To
see this, notice that the hypercube $2^n$ is path connected, where a path
connecting two functions $f, g \in 2^n$ is a sequence of `bit flips', or
functions 
\[
  f = z_0, z_1, \dots , z_{e-1} = g
\]
such that two consecutive functions differ in exactly one argument. 

\begin{claim}
Let $f, g \in 2^n = 2^{(n+1) \setminus\{n\}}$ be functions that differ in
exactly one argument. If the $(n)$-support line $(c_{n-1})_f$ is constant
then the $(n)$-support line $(c_{n-1})_g$ is also constant.
\end{claim}
\begin{claimproof}
Suppose that $k\in n$ is the unique argument such that $f(k) \neq g(k)$. We
may assume without loss of generality that
\[
  f(k) = 0,
  \qquad
  g(k) = 1,
  \qquad
  (c_{n-1})_f = \LinePic[a][a][0.75],
  \qquad \text{and} \qquad
  (c_{n-1})_g = \LinePic[r^j_{4i+\epsilon}][r^j_{4i+\tau}][0.75]
\]
for some $a \in A_n$, $i, j \in \omega$, and $\epsilon, \tau \in \{0,2\}$.
Since $f$ and $g$ agree everywhere on $n \setminus\{k\}$, there is $h
\in 2^{(n+1) \setminus\{k,n\}}$ such that $f$ and $g$ extend $h$. This
means that $(c_{n-1})_f$ and $(c_{n-1})_g$ will be the columns of the
$(k,n)$-cross section square
\[
  (c_{n-1})_h =  \Square[a][a][r^j_{4i+\epsilon}][r^j_{4i+\tau}].
\]
Applying Lemma \ref{lem:must_be_sucessors_An} to this we obtain $\epsilon =
\tau$, which proves the claim.
\end{claimproof}

An induction using the above claim shows that if $(c_{n-1})_f$ is a constant
$(n)$-support line and $g$ is connected to $f$ by a path in $2^n$ then
$(c_{n-1})_g$ is also a constant $(n)$-support line. Since $2^n$ is path
connected, this forces every $(n)$-support line of $c_{n-1}$ to be constant.
The inductive hypothesis applied to $c_{n-1}$ now implies that the
$(n)$-pivot line of $c_{n-1}$ is also constant, which finishes the proof in
this case. 

\Case{$\SLines_n(c_{n-1}) \cap D = \emptyset$}
The condition for this case is equivalent to the statement
\[
  \SLines_n(c_{n-1}) \subseteq \Bigg\{
    \LinePic[r^j_{4i+\epsilon}][r^j_{4i+\tau}][0.75]
    : i,j\in \omega,\ \{\epsilon, \tau\} = \{0,2\} \Bigg\}.
\]
For $f \in 2^n$ the $(n)$-support line of $h$ at $f$ is therefore
\[
  h_f 
    = t\big( (c_0)_f, \dots, (c_{n-2})_f, (c_{n-1})_f \big) 
    = t \Bigg( \LinePic[a_0][a_0][0.75], \dots,
      \LinePic[a_{n-2}][a_{n-2}][0.75],
      \LinePic[r^j_{4i+\epsilon}][r^j_{4i+\tau}][0.75] \Bigg)
\]
for some $a_0, \dots, a_{n-2} \in A_n$, $i,j \in \omega$, and $\{\epsilon,
\tau\} = \{0,2\}$. By assumption $h_f$ is constant, so an application of
Lemma \ref{lem:trivial_An} yields $a_0, \dots, a_{n-2} \in R$. This
reasoning works for any $f\in 2^n \setminus \{\textbf{1}\}$, so we conclude
that 
\[
  \SLines_n(c_d) \subseteq \Bigg\{ \LinePic[ r'][r''][0.75] : r',r''\in R \Bigg\}
\]
for all $d \in n$. Applying Lemma \ref{lem:must_be_generators_ultimate} to
this we obtain that $c_d \in X_0$ for all $d \in n$. 

The situation now is that all of the $c_0, \dots, c_{n-1}$ are generators of
$\M(1,\dots, 1)$. For $d\neq n-1$ we know that $\SLines_n(c_d)$ is constant,
so it follows that
\[
  c_d = \gcube_i^{n+1} (a_d, b_d)
\]
for some $i \neq n$ and $a_d, b_d \in A_n$. We have also assumed (in the
paragraph before the case analysis) that the $(n)$-pivot line of $c_{n-1}$
is not constant, so we also know that 
\[
  c_{n-1} = \gcube_{n}^{n+1}(r^j_{4i+\epsilon}, r^j_{4i + \tau})
\]
for some $i,j \in \omega$ and $\{\epsilon, \tau\} = \{0,2\}$. 

Each of the $c_d$ with $d\neq n-1$ is an $(n+1)$-dimensional cube that is
constant in all but a single dimension in $(n+1)\setminus \{n\}$. There are
$n-1$ many such $c_d$ and $n+1$ many dimensions, so there is at least one
$k\in (n+1)\setminus \{n\}$ such that all of the $c_d$ are constant in
dimension $k$. It follows that, for each $d \neq n-1$, any $(k,n)$-cross
section square of $c_d$ is constant. In particular, for $\textbf{1} \in
2^{(n+1) \setminus \{k,n\}}$
\[
  (c_d)_\textbf{1} = \Square[b_d][b_d][b_d][b_d]
\]
for some $b_d\in A_n$. Hence, the $(k,n)$-pivot square of $h$ is 
\begin{multline*}
  h_\textbf{1} 
    = t \Bigg( 
      \Square[b_0][b_0][b_0][b_0],
      \dots,
      \Square[b_{n-2}][b_{n-2}][b_{n-2}][b_{n-2}],
      \SquareXY[r^j_{4i+\epsilon}][r^j_{4i+\tau}][r^j_{4i+\epsilon}][r^j_{4i+\tau}][1.25]
    \Bigg) \\
    = \SquareXY[t(b_0, \dots, b_{n-2}, r^j_{4i+\epsilon})][t(b_0, \dots, b_{n-2}, r^j_{4i+\tau})][t(b_0, \dots, b_{n-2}, r^j_{4i+\epsilon})][t(b_0, \dots, b_{n-2}, r^j_{4i+\tau})][3.5].
\end{multline*}
One of the columns of the $(k,n)$-pivot square of $h$ is an $(n)$-support
line and the other is the $(n)$-pivot column. Since these columns are equal
and we have assumed that the $(n)$-support lines of $h$ are constant, it
follows that the $(n)$-pivot line is constant as well.

This completes the case analysis. In all cases, we showed that if all the
$(n)$-support lines of $h$ are constant, then the $(n)$-pivot line of $h$ is
constant as well. From the remarks at the start of the proof, this is enough
to show that $\A_n$ is $(n)$-step supernilpotent.
\end{proof}  

\section{Concluding Remarks}  
A manuscript in preparation by the second author shows that supernilpotence
implies nilpotence in varieties that satisfy a nontrivial idempotent
equational condition. Such varieties are called Taylor varieties in the
literature. In \cite{olsak}, Ol\v{s}\'ak produces a strong Mal'cev condition
characterizing the class of Taylor varieties. We ask the question: If
$[\var]$ is a chapter in the lattice of interpretability of types that does
not lie in the interval above Ol\v{s}\'ak's term, must there be variety
$\cl{W} \in [\var]$ with a supernilpotent algebra that is not nilpotent?

\section{Acknowledgements} \label{sec:acknowledge} 
The authors would like to thank the following people:
\begin{itemize}
  \item Jakub Bulin and Alexandr Kazda, for hosting both
    authors at the Charles University Mathematics Department, where the
    first of these results were obtained;
  \item Alexandr Kazda and Keith Kearnes, for useful discussions relating to
    this topic.
\end{itemize}

\bibliographystyle{amsplain}   
\bibliography{refs.bib}
\begin{center}
  \rule{0.61803\textwidth}{0.1ex}   
\end{center}
\end{document}